\documentclass[a4paper, 11pt]{amsart}

\usepackage{enumerate}
\usepackage{mathrsfs}
\usepackage{amssymb}
\usepackage[all]{xy}

\title[]{Quotients of spectra of almost factorial domains and Mori dream spaces}
\author{Marcel Maslovari\'c*} 

\thanks{*Supported by the DFG Research Training Group 1493 "Mathematical Structures in Modern Quantum Physics"}

\date{\today}

\address{Marcel Maslovari\'c,
Mathematisches Institut,
Georg-August-Universit\"at G\"ottingen,
Bunsenstra\ss e 3-5, 
D-37073 G\"ottingen,
Germany}
\email{mmaslov@uni-goettingen.de}

\newcommand{\N}{{\mathbb{N}}}
\newcommand{\Z}{{\mathbb{Z}}}
\newcommand{\Q}{{\mathbb{Q}}}
\newcommand{\R}{{\mathbb{R}}}

\newcommand{\Ocal}{{\mathcal{O}}}
\newcommand{\G}{{\mathbb{G}}}
\newcommand{\A}{{\mathbb{A}}}
\renewcommand{\P}{{\mathbb{P}}}
\newcommand{\C}{{\mathcal{C}}}
\newcommand{\M}{{\mathcal{M}}}
\newcommand{\Zcal}{{\mathcal{Z}}}
\newcommand{\F}{{\mathbb{F}}}

\newcommand{\Eff}{{\overline{\mathrm{Eff}}}}

\newcommand{\Mov}{{\overline{\mathrm{Mov}}}}
\newcommand{\Nef}{{\mathrm{Nef}}}
\newcommand{\lcm}{{\mathrm{lcm}}}
\newcommand{\ord}{{\mathrm{ord}}}

\newcommand{\Mat}{{\mathrm{Mat}}}
\newcommand{\GL}{{\mathrm{GL}}}

\newcommand{\p}{{\mathfrak{p}}}
\newcommand{\eps}{{\varepsilon}}

\newtheorem{prop}{Proposition}[section]
\newtheorem{lemma}[prop]{Lemma}
\newtheorem{cor}[prop]{Corollary}
\newtheorem{thm}[prop]{Theorem}
\theoremstyle{definition}
\newtheorem{defin}[prop]{Definition}
\newtheorem{rem}[prop]{Remark}
\newtheorem{conv}[prop]{Convention}
\newtheorem{example}[prop]{Example}

\begin{document}
 
\begin{abstract}
We prove that a GIT chamber quotient of an affine variety $X=Spec(A)$ by a reductive group $G$, where $A$ is an almost factorial domain, is a Mori dream space if it is projective, regardless of the codimension of the unstable locus. This includes an explicit description of the Picard number, the pseudoeffective cone, and the Mori chambers in terms of GIT.\\
We apply the results to quiver moduli, to show that they are Mori dream spaces if the quiver contains no oriented cycles, and if stability and semistability coincide. We give a formula for the Picard number in quiver terms. As a second application we prove that geometric quotients of Mori dream spaces are Mori dream spaces as well, which again includes a description of the Picard number and the Mori chambers.\\
Some examples are given to illustrate the results.
\end{abstract}

\maketitle

\tableofcontents

\section{Introduction}

It has been widely observed that variation of Geometric Invariant Theory (VGIT) and birational geometry are closely related. In \cite{hukeel}, Hu and Keel defined Mori dream spaces as a very particular class of varieties where this relation is very close, and where the Mori program works very well. Under some nice properties (see \cite{hukeel} Lemma 2.2 for the precise requirements) they proved that a quotient of an affine variety (over some algebraically closed field $k$) with respect to a GIT chamber is a Mori dream space, that the Picard group of the quotient is identified with the character group of the acting group, and that GIT chambers are identified with equivalence classes of divisors on the quotient, the so called Mori chambers (consider Theorem 2.3 in \cite{hukeel}).\\
One of the properties they required is that the unstable locus be of codimension at least two, which seems to be a technical hurdle in the application of this theorem. In the present paper we generalize the result by Hu and Keel to arbitrary GIT chambers, regardless of the codimension of the unstable locus.\\

Our first result is the following. The required definitions are given in sections \ref{S:quotientafd} and \ref{S:preliminaries}, and the proof is given in section \ref{S:quotientafd} as well. We recall that $E(X)$ is the group of invertible regular functions on $X$ modulo scalars, and that for $U\subset X$ the set $\Zcal(U)$ consists of the irreducible components of $X\setminus U$ which are of codimension one. If $U$ is invariant under the action of a group $G$ on $X$, then there is an induced action of $G$ on $\Zcal(U)$.

\begin{thm}\label{T:quotientafd_intro}
Let a reductive group $G$ act on a variety $X=Spec(A)$, where $A$ is an almost factorial domain such that $A^G=k$. Suppose that $U\subset X$ is the stable locus of a character $\chi_0\in\chi(G)$ such that stability and semistability coincide, and denote by $q:U\to Y$ the associated geometric quotient.
\begin{enumerate}
\item All rational contractions of $Y$ are induced by GIT, and the variety $Y$ is a Mori dream space.
\item The canonical map
$$\psi:\chi(G)_\Q\to Pic(Y)_\Q$$
is surjective with kernel of rank $rk(E(X)^G)+|\Zcal(U) /G|$.
\item The pseudoeffective cone of $Y$ is precisely the image of all stable $G$-ample classes in $\chi(G)_\Q$.
\item Via $\psi$ the Mori chambers of $Y$ and the stable GIT chambers in $\chi(G)_\Q$ are identified.
\end{enumerate}
\end{thm}

Remark \ref{R:allquotientsgit} shows that the apparent restriction to geometric GIT quotients is actually not a restriction.\\

A natural question to ask is whether a quotient $U\to Z$, where $U\subset Y$ is an open subset of a Mori dream space, is again a Mori dream space. This question was answered in the positive by \cite{baeker}, but a description of the relevant data, like the Picard number or Mori chambers, is still missing.\\
As an application of Theorem \ref{T:quotientafd_intro} we give another proof, using techniques different from \cite{baeker}, which is given in section \ref{S:quotientofmds}, and includes a description of the Picard number and Mori chambers. Again, by Remark \ref{R:allquotientsgitmds}, there is no restriction in assuming that this quotient be given by a GIT chamber, but it seems more natural to state the result for arbitrary quotients nonetheless. 

\begin{thm}
Let a reductive group $G$ act on a Mori dream space $Y$ such that there exists a geometric quotient $U\to Z$ of an open and $G$-invariant subset $U\subset Y$, where $Z$ is projective and stabilizers for points on $U$ are finite. Then $Z$ is a Mori dream space with Picard number
$$\rho(Z)=\rho(Y)+rk(\chi(G))-|\Zcal(U)/G|.$$
Interpreted properly, the intersections of Mori chambers of $Z$ with $Pic(Y)_\Q$ form a refinement of a certain subsystem of Mori chambers of $Y$ (for the precise statement consider Theorem \ref{T:quotientofmds}).
\end{thm}

Since quiver moduli are GIT quotients of affine space, one might expect that they are Mori dream spaces under some mild assumptions. Indeed, an application of the result of \cite{baeker} directly confirms this expectation, but does not provide any information on the Picard group or the Mori chambers (consider Remark \ref{R:quivermdsalternative}).\\
The attempts to apply Theorem 2.3 of \cite{hukeel} to the quiver situation seem to be hampered by the codimension restriction - see e.g. the discussion of quiver flag varieties in \cite{craw}, in particular the discussion of the occurrence of unstable codimension one components in the proof of Proposition 3.1.\\
In section \ref{S:quiver} we use our extended result to give a quick proof that certain quiver moduli are Mori dream spaces, and give a formula for their Picard number as well as a description of the Mori chambers.

\begin{thm}
Suppose that $Q$ is a quiver without oriented cycles, that $d$ is a dimension vector, and that $\theta$ is a stability condition such that stability and semistability coincide.\\
Then $\M_{d,\theta}(Q)$ is a Mori dream space satisfying the assertions of Theorem \ref{T:quotientafd_intro}. In particular, the Picard number is given as
$$\rho(\M_{d,\theta}(Q))=|Q_0|-(r+1),$$
where $r$ is the number of components of the unstable locus which are of codimension one.\\
\end{thm}
Since the GIT chambers of quiver moduli can be treated combinatorically, at least in specific examples, we expect that this might provide an interesting class of examples for Mori dream spaces. Conversely, it might be helpful in the application of birational methods to the representation theory of quivers (or artinian algebras).\\

It might be possible to extend the results of this document to good quotients in future work. Another potential application is the problem of finding Mori dream subspaces, since the known results are heavily restricted by assumptions on the unstable locus (see e.g. \cite{jow}).\\

\noindent {\bf Acknowledgement}. 
The author thanks David Schmitz for helpful discussions and feedback to an earlier draft of this document, and Markus Reineke for answering questions about quiver moduli. Special thanks belongs to Henrik Sepp\"anen for his patient guidance and many fruitful meetings.

\section{Preliminaries \label{S:preliminaries}}

Let $k$ denote an algebraically closed field. We assume all varieties to be defined over $k$ and to be irreducible, unless stated otherwise.

Given any variety $X$ we define the group
$$E(X)=\Ocal(X)^*/k^*.$$
This group is free and finitely generated by Proposition 1.3 of \cite{kkv}.\\

We will often use a variant of the Hartogs lemma for normal varieties (see Theorem 6.45 in \cite{goertzwedhorn}).
\begin{lemma}\label{L:hartogs}
Let $U\subset X$ denote an open subset in a normal variety such that $X\setminus U$ is of codimension greater or equal to $2$. Then the restriction map
$$\Ocal(X)\to \Ocal(U)$$
is an isomorphism.
\end{lemma}
~\\

\subsection{Rational maps and Mori dream spaces}~\\

To clarify the notation used in the present paper, we give some basic definitions and conventions. The notions of rational contractions, SQMs and Mori dream spaces are given as in the foundational paper \cite{hukeel}. We sometimes use equivalent descriptions provided by \cite{casagrande}.\\

The domain $dom(f)$ of a rational map $f:X\dasharrow Y$ is the maximal open subset in $X$ such that there exists a regular representative
$$f':dom(f)\to Y.$$
The image of $f$ is defined as $im(f)=im(f')$, and if $V\subset Y$ then
$$f^{-1}(V)=(f')^{-1}(V)\subset dom(f).$$
The exceptional locus $exc(f)$ is defined as the complement of the set of points $x\in dom(f)$ with the property that there exists an open neighborhood $f(x)\in V\subset Y$ such that the restriction $f':f^{-1}(V)\to V$ is an isomorphism.\\
Two rational maps $f:X\dasharrow Y$ and $g:X\dasharrow Y'$ are called Mori equivalent, if there is an isomorphism between their images such that the obvious diagram commutes as rational maps.\\

Let $f:X\dasharrow Y$ denote a dominant rational map of projective and normal varieties, where additionally $X$ is $\Q$-factorial. Then $f$ is called a rational contraction if there exists a resolution
$$\xymatrix{&W\ar[ld]_p\ar[rd]^q& \\
X \ar@{-->}[rr]_f &&Y,}$$
where $W$ is smooth and projective, $p$ is birational, and for every $p$-exceptional effective divisor $E$ on $W$ the equation
$$q_*(\Ocal_W(E))\simeq \Ocal_Y$$
holds. For a $\Q$-Cartier divisor $D$ on $Y$ the pullback is defined as
$$f^*(D)=p_*(q^*(D)).$$
All these notions do not depend on the choice of a resolution, and the pullback is well-defined on numerical equivalence classes.\\

The map $f$ is called a small $\Q$-factorial modification (SQM) if $Y$ is $\Q$-factorial and $f$ is an isomorphism in codimension one. In that case $f$ is automatically a rational contraction (consider Remark 2.2 in \cite{casagrande}).

\begin{defin}
A projective and normal variety $X$ is called a Mori dream space if the following holds.
\begin{enumerate}
\item $X$ is $\Q$-factorial, and $Pic(X)_\Q=N^1(X)_\Q$.
\item \label{I:nefcone} The nef cone $\Nef(X)$ is spanned by finitely many semiample line bundles.
\item There is a finite collection of SQMs $f_i:X\dasharrow X_i$ such that $X_i$ satisfies (\ref{I:nefcone}) and the moving cone $\Mov(X)$ is the union of the cones $f_i^*(\Nef(X_i))$.
\end{enumerate}
\end{defin}

Let $D$ denote a Cartier divisor on a projective and normal variety $X$. The section ring of $D$ is defined as
$$R(X,D)=\bigoplus_{n\geq 0}H^0(X,nD)=\bigoplus_{n\geq 0}H^0\left(X,\Ocal_X(nD)\right).$$
If $R(X,D)$ is finitely generated, then evaluation of sections gives rise to a rational map
$$f_D:X\dasharrow Y_D=Proj(R(X,D)),$$
which is regular outside the base locus of $D$. Up to multiples, these maps are rational contractions (see Lemma 1.6 in \cite{hukeel}). Two divisors are called Mori equivalent if their associated rational maps are Mori equivalent.\\

\subsection{Geometric Invariant Theory} \label{SS:git}~\\
We recollect basic definitions of Geometric Invariant Theory, mainly to fix notation. For the relevant aspects of VGIT we refer to \cite{dolghu} or \cite{thaddeus}, or to \cite{halic} for the transfer to the affine case.\\

Let a reductive group $G$ act on a variety $X$.\\
A $G$-line bundle $E\to X$ on $X$ is a line bundle $E\to X$, equipped with an action $G\times E\to E$ such that the diagram
$$\xymatrix{G\times E \ar[r] \ar[d] & E \ar[d] \\
G\times X \ar[r] & X}$$
commutes, and the induced action on fibres of $E$ is linear. By $H^0(X,E)^G$ we denote the invariant sections, and by $R(X,E)^G$ the section ring of invariant sections. Two $G$-line bundles $E,E'$ are isomorphic if there exists an equivariant isomorphism of line bundles $E\to E'$ (leaving the base space fixed). By $Pic^G(X)$ we denote the group of isomorphism classes of $G$-line bundles on $X$. Interpreted correctly, $Pic^\bullet(\bullet)$ is functorial (see Lemma \ref{L:picfunctor} below).\\

The semistable locus with respect to $E$ is defined as
$$X^{E-sst}=\bigcup_{f}D(f),$$
where $D(f)=X\setminus N(f)$, and the union is taken over all $f\in R(X,E)^G$ such that $D(f)$ is affine. The stable locus $X^{E-st}$ is the set of points $x\in X^{E-sst}$ such that the stabilizer $G_x$ is finite, and the orbit $G*x$ is closed in the semistable locus. Both sets could possibly be empty.\\
The evaluation of sections defines a good quotient
$$q_E:X^{E-sst}\to Y_E=Proj(R(X,E)^G),$$
which restricts to a geometric quotient on the stable locus.\\

In the following we will denote by $L\to X$ the trivial line bundle, and if $\chi\in\chi(G)$ is a character, by $L_\chi$ the corresponding $G$-line bundle with action
$$g*(x,e)=(g*x,\chi(g)\cdot e).$$
A section $f\in H^0(X,L_\chi)^G$ is called a semiinvariant and satisfies
$$g*f=\chi(g)\cdot f.$$
Note that some authors call a function as given above a semiinvariant of rank $1$. We do not need this distinction.\\
We abbreviate $X^{\chi-sst}=X^{L_\chi-sst}$, and similarly for the stable locus. The associated good quotient is denoted by $q_\chi:X^{\chi-sst}\to Y_\chi$.\\

Two characters are called GIT equivalent if the semistable loci are the same. A character, or a GIT class, is called $G$-ample if the semistable locus is nonempty. We will not make an explicit distinction between GIT classes and their images in $Pic^G(X)_\Q$ or $Pic^G(X)_\R$.\\

\begin{lemma}\label{L:picfunctor}
Consider the category $C$ where objects are pairs $(X,G)$ of algebraic groups acting on varieties $X$, and morphisms
$$(f,\varphi):(X,G)\to (Y,H)$$
are pairs of morphisms of varieties and algebraic groups respectively such that $f(g*x)=\varphi(g)*f(x)$ for all $x\in X$ and $g\in G$. Then there is a functor
$$Pic^\bullet(\bullet):C\to Ab,$$
extending the usual functor $Pic(\bullet)$. If $(f,\varphi)$ is a morphism in $C$ as above, and $\chi\in\chi(H)$, then $(f,\varphi)^*(L_\chi)=L_{\varphi^*(\chi)}$.

\begin{proof}
We use the notation $(f,\varphi)$ for a morphism in $C$ as in the statement of the lemma. Recall that the pullback of an ordinary line bundle $\pi:E\to Y$ is given as
$$f^*(E)=\left\{(x,e)\in X\times E | f(x)=\pi(e)\right\},$$
and if additionally $E$ is an $H$-line bundle, we define $g*(x,e)=(g*x,\varphi(g)*e)$. Now, it is straightforward to check that this construction is well-defined and satisfies the assertions.
\end{proof}
\end{lemma}

For projective $X$, the results of \cite{dolghu} tell us that there are finitely many GIT classes, which additionally are rational polyhedral. This result was transferred to the case of a $G$-module by \cite{halic}. If $X$ is any normal affine variety, we can use the standard linearization
$$X\hookrightarrow \A^N$$
into a $G$-module $\A^N$ to give the description $X^{\chi-sst}=X\cap\left(\A^N\right)^{\chi-sst}$, and a similar description for the stable loci. Thus, at least if all $G$-line bundles on $X$ are up to multiples of the form $L_\chi$, we have the same chamber behavior for the action of $G$ on $X$.\\
By a GIT chamber we mean a GIT class whose relative interior is open, or the closure thereof. We will assume henceforth that stability and semistability with respect to a chamber coincide. Note that this may fail for all GIT chambers at once, or it may fail for some chambers, while it holds for others (consider the counterexample \cite{ressayre}). However there are two important situations where our assumption is satisfied.
\begin{enumerate}
\item This holds for the action of $PG_d$ on $R_d(Q)$, where $Q$ is a quiver (for the definition see Section \ref{S:quiver}), and $d$ is a coprime dimension vector, i.e. the entries of $d$ admit no nontrivial common divisor (consider Section 3.5 in \cite{reineke}).
\item If $G=T$ is a torus, the fact that stability and semistability coincide for one chamber implies the same assertion for the other chambers (again we use a linearization to reduce to the case of a $T$-module, where it holds by Proposition 3.10 in \cite{halic}).
\end{enumerate}
Furthermore, we can use the linearization to transfer the Hilbert-Mumford criterion for $G$-modules, as established in \cite{halic} or \cite{king}, to an arbitrary normal affine variety.\\

Suppose that we are given two characters $\chi,\chi'\in\chi(G)$, and assume that $X^{\chi-st}$ is nonempty. Setting
$$V=X^{\chi-st}\cap X^{\chi'-sst},$$
the restriction $V\to q_\chi(V)\subset Y_\chi$ is again a geometric quotient, with open image. Further, the composition
$$V\subset X^{\chi'-sst}\to Y_{\chi'}$$
is $G$-invariant, and hence factors to give a morphism $q_\chi(V)\to Y_{\chi'}$. This defines a rational map $f:Y_\chi\dasharrow Y_{\chi'}$.\\
The situation is summarized in the following diagram.
$$\xymatrix{X^{\chi-sst} \ar[d]_{q_\chi} & V \ar[d] \ar@{_(->}[l] \ar@{^(->}[r] & X^{\chi'-sst} \ar[d]^{q_{\chi'}} \\
Y_\chi & q_\chi(V) \ar@{_(->}[l] \ar[r] & Y_{\chi'}}$$
In Lemma \ref{L:rationalmaps} we give a partial description of the domain and exceptional locus of this map.\\
Note that $f$ is a birational map if $X^{\chi'-st}\neq \emptyset$.\\
Indeed, the images under $q_\chi$ and $q_{\chi'}$ of
$$V'=X^{\chi-st}\cap X^{\chi'-st}$$
give open subsets $q_\chi(V')\subset Y_\chi$ and $q_{\chi'}(V')\subset Y_{\chi'}$, which are both geometric quotients of $V'$, and are hence isomorphic.\\

\subsection{Almost factorial domains}\label{SS:afds}~\\
Let $X=Spec(A)$ denote an affine variety. In particular, $X$ is supposed to be irreducible, so $A$ is an integral domain.\\
It is well-known that $A$ is a UFD if and only if $X$ is normal and $Cl(X)=0$. Following \cite{storch}, we call $A$ an almost factorial domain (AFD) if and only if $X$ is normal and $Cl(X)$ is torsion.\\

For the ease of the reader we review the relevant properties of AFDs. We start with properties and conventions which are established in \cite{storch}.
\begin{enumerate}
\item A nonzero nonunit $x\in A$ is called primary if the ideal $(x)\subset A$ is a primary ideal. In that case, the associated prime ideal $\p_x$ is defined as the prime ideal associated to $(x)$.
\item Two primary elements which have the same associated prime ideal are up to units powers of a third primary element. Conversely, powers of a primary element are again primary.
\item For a nonzero nonunit $y\in A$ a suitable power $y^n$ can be factored into primary elements.
\item If $D\subset X$ is a prime divisor, there exists a primary element $x\in A$ and a natural number $n$ such that
$$nD=div(x).$$
\end{enumerate}

Additionally, we need the uniqueness of the associated prime ideals in a primary decomposition, which can be proven by a slight variation of the well-known proof that the elements in a decomposition into primes are uniquely determined. We note that the prime ideal associated to a primary ideal is given as its radical ideal.

\begin{lemma}\label{L:uniquenessassprimes}
The following statements hold in any integral domain.
\begin{enumerate}
\item Suppose there are two decompositions
$$a_1\cdot\ldots\cdot a_r=b_1\cdot\ldots\cdot b_s$$
into primary elements. Then the sets of the prime ideals associated to the $a_i$ and associated to the $b_j$ respectively coincide.
\item If we have a primary element $x\in A$ such that
$$x\mid a_1\cdot\ldots\cdot a_r,$$
then $x|a_i^n$ for a suitable index $i$ and a suitable power $n$.
\item Given two primary elements $x,y\in A$ such that $x\mid y$, it follows that the associated prime ideals coincide.
\end{enumerate}

\begin{proof}
The first assertion easily follows from the second and third assertions.\\

Induction over $r$ proves the second assertion, where $r=1$ is trivial. If $x\mid a_1\cdot\ldots\cdot a_{r-1}$, we are done by the induction hypothesis, and if not we have $a_r^n\in (x)$ since $x$ is primary.\\

Under the hypothesis of the third assertion we may write $y=xy'$. Now $y\nmid y'$, since otherwise $x$ would be a unit, so $y\mid x^n$ for some power $n$. Thus there are inclusions of ideals
$$(y)\subset (x) \supset (x^n) \subset (y).$$
This yields the inclusions
$$\p_y\subset \p_x=\p_{x^n}\subset \p_y$$
by taking the radicals, which finishes the proof.
\end{proof}
\end{lemma}

\section{Quotients of spectra of AFDs \label{S:quotientafd}}

As a first step we compute the regular invertible functions on an open subset of the spectrum of an AFD. Obviously, this does not require the introduction of a group action yet.

\begin{conv}\label{generalsituation}
Let $A$ denote an AFD, and let $U\subset X=Spec(A)$ denote an open subset. By
$$\Zcal=\Zcal(U)=\{Z_1,\ldots,Z_q\}$$
we denote the set of irreducible components of $X\setminus U$ which are of codimension one in $X$. Interpreting the elements $Z_i$ as prime divisors on the variety $X$, the fact that $A$ is almost factorial implies that $n_iZ_i=div(g_i)$ is a principal divisor for some natural number $n_i$ and a primary element $g_i\in A$ (consider section \ref{SS:afds}). In particular it holds that $Z_i=N(g_i)$.\\
In the special case where $A$ is a UFD, we have that every prime divisor is principal, and hence $Z_i=div(g_i)$ for prime elements $g_i\in A$.
\end{conv}

\begin{lemma}
\label{L:invertiblefunctions}
In the situation of Convention \ref{generalsituation} the map
$$F:A^*\times\Z^q\to \Ocal(U)^*,$$
given by
$$(\lambda,a_1,\ldots,a_q)\mapsto \lambda\cdot g_1^{a_1}\cdot\ldots\cdot g_q^{a_q}|_U,$$
is injective with torsion cokernel. In particular, there is an isomorphism
$$E(X)_\Q\times \Q^q\to E(U)_\Q.$$

If $A$ is a UFD the map $F$ is an isomorphism, and $E(X)\times \Z^q\simeq E(U)$.
\end{lemma}

\begin{proof}
First assume that
$$1=\lambda\cdot g_1^{a_1}\cdot\ldots\cdot g_q^{a_q}$$
on $U$ for some $(a_1,\ldots,a_q)\in\Z^q$, and $\lambda\in A^*$ on $U$. By removing factors with $a_i=0$ and bringing factors with $a_i<0$ to the other side, we can assume without loss of generality
$$g_1^{a_1}\cdot \ldots \cdot g_s^{a_s}=\lambda\cdot g_{s+1}^{a_{s+1}}\cdot\ldots\cdot g_r^{a_r}$$
for some $a_i>0$ globally on $X$.\\
If we assume that this equation is non-trivial, which is $r\neq 0\neq s$, there exists a point $x\in Z_1\setminus (Z_{s+1}\cup\ldots\cup Z_r)$. But then we obtain the contradiction
$$0=g_1(x)^{a_1}\cdot\ldots\cdot g_s(x)^{a_s}=\lambda(x)\cdot g_{s+1}(x)^{a_{s+1}}\cdot\ldots\cdot g_r(x)^{a_r}\neq 0.$$
Thus all $a_i$ have to vanish, which also implies $\lambda=1$.\\
This proves the injectivity of the map $F$.\\

To prove surjectivity (up to torsion in the case where $A$ is an AFD), we may assume without loss of generality that $X\setminus U$ is of pure codimension one using the Hartogs lemma \ref{L:hartogs}.\\

Given $g,h\in\Ocal(U)=\Ocal(D(g_1\cdot\ldots\cdot g_q))$, which are inverse to each other, we may write
$$g=a_g/(g_1\cdot\ldots\cdot g_q)^{s_g},~ h=a_h/(g_1\cdot\ldots\cdot g_q)^{s_h}$$
for some $a_g,a_h\in A$ and $s_g,s_h\in \N$, and hence
$$(g_1\cdot \ldots \cdot g_q)^{s_g+s_h}=a_ga_h.$$

First assume that $Cl(X)=0$. Then, since the $g_i$ are irreducible in the unique factorization domain $A$, both $a_g$ and $a_h$ are products of multiples of the $g_i$ up to invertible elements, and are hence contained in the image of $F$, so $g$ is in the image of $F$ as well.\\

For the general case we choose a power $n$ such that there exist decompositions $a_g^n=x_1\cdot\ldots\cdot x_s$ and $a_h^n=y_1\cdot\ldots\cdot y_t$ into primary elements. Clearly
$$(g_1\cdot \ldots\cdot g_q)^{n(s_g+s_h)}=a_g^na_h^n=x_1\cdot\ldots\cdot x_sy_1\cdot\ldots\cdot y_t$$
are two decompositions of $a_g^na_h^n$ into primary elements. By uniqueness of the associated prime ideals in a primary decomposition (Lemma \ref{L:uniquenessassprimes}), the prime ideals associated to primary elements on the right must be contained in the family $\p_1,\ldots,\p_r$. Thus, up to units each of the elements $x_i$ is a multiple of a primary element $z_i$ such that one of the $g_j$ is a multiple of $z_i$. Again taking multiples if necessary, this implies that some power of $a_g$ is up to units a product of the $g_i$, which implies that $g^N$ is in the image of $F$ for some power.
\end{proof}

Now we introduce a group action. Note that any geometric quotient of the form as below is given as a GIT quotient for a character in the interior of a GIT chamber by Remark \ref{R:allquotientsgit}.

\begin{conv}\label{groupsituation}
Given the situation of Convention \ref{generalsituation}, we assume furthermore that a reductive group $G$ acts on $X$ such that $U$ is $G$-invariant.\\
Later on, we will further require that there exists a geometric quotient
$$q:U\to Y,$$
and that the stabilizers $G_x$ for $x\in U$ are finite. Except for in Proposition \ref{P:picardnumber}, we additionally assume that $Y$ is projective, and hence $E(Y)=0$. All GIT quotients of $X$ are required to be projective, which is implied by $\Ocal(X)^G=A^G=k$.
\end{conv}

We recall that the vanishing order of $g_i$ on $Z_i=N(h_i)$ is denoted by $n_i$.

\begin{lemma}\label{L:actiononZ}
In the situation of Convention \ref{groupsituation} the following holds.
\begin{enumerate}
\item The action of $G$ on $X$ induces an action of $G$ on $\Zcal$, and we denote the set of orbits as
$$\Zcal/G=\{B_1,\ldots,B_r\}.$$
If $G$ is connected, this action is trivial.
\item Suppose that $N(h)\subset X$ is $G$-invariant for some nonzero nonunit $h\in A$. We further require that in the primary decomposition
$$h^n=h_1\cdot\ldots\cdot h_s$$
the vanishing orders of $h_i$ and $h_j$ on their respective nullstellensets $N(h_i)$ and $N(h_j)$ agree, whenever $N(h_i)$ and $N(h_j)$ have the same orbit under the action of $G$ on $\Zcal(D(h))$, and that the prime ideals associated to $h_i$ and $h_j$ differ when $i\neq j$. Then $h^n$ is a semiinvariant with respect to a uniquely determined character. Conversely, all semiinvariants are of such a form.
\item For any orbit $B_i\in\Zcal/G$ the function
$$f_i=\prod_{Z_j\in B_i}g_j^{N_{ij}},$$
where $N_{ij}=\lcm(n_k|Z_k\in B_i)/n_j$, satisfies the assertions of (2), so $g*f_i=\chi_i(g)\cdot f_i$ for a uniquely determined character $\chi_i\in\chi(G)$.
\item The isomorphism $F:E(X)_\Q\times \Q^q\to E(U)_\Q$ given in Lemma \ref{L:invertiblefunctions} induces an isomorphism
$$E(X)^G_\Q\times \Q^{r}\to E(U)^G_\Q.$$
If $A$ is a UFD, the same statement is true with coefficients in $\Z$.
\end{enumerate}
\end{lemma}

\begin{proof}
To prove the first assertion observe that if $G'\subset G$ is a connected component, and $Z\in\Zcal$, then $G'*Z\subset X\setminus U$ is irreducible and hence contained in some $Z'\in \Zcal$. By equality of dimensions we have $G'*Z=Z'$. It is easy to see that this defines an action of $G/G^0$ on $\Zcal$, and hence an action of $G$.\\

For the second assertion we need some preparations.\\
By the first assertion $G$ acts on $\Zcal(D(h))=\{N(h_1),\ldots,N(h_s)\}$, and we denote by $\pi_g$ the permutation of indices associated to an element $g\in G$. It is immediate to verify that $N(g*h_i)=N(h_{\pi_g(i)})$, so
$$g*h_i=\lambda z^k ~\mathrm{and}~ h_{\pi_g(i)}=\lambda' z^l$$
are multiples of a third primary element $z\in A$ up to units $\lambda,\lambda'\in A^*$.\\
On the other hand, $g$ can be interpreted as an automorphism of $A$, and hence of $X$, which induces an isomorphism of local rings
$$\Ocal_{X,\xi_i}\to \Ocal_{X,\xi_{\pi_g(i)}},$$
where $\xi_i$ and $\xi_j$ are the generic points of $N(h_i)$ and $N(h_j)$ in the affine scheme $X$. So $\ord_{N(h_i)}(h_i)=\ord_{N(h_{\pi_g(i)})}(g*h_i)$, from which we deduce $k=l$. In other words, $g$ permutes the elements $h_i$ up to units, so that the equation
$$g*h_1\cdot\ldots\cdot h_s=\lambda(g)\cdot h_1\cdot\ldots\cdot h_s$$
holds for a unit $\lambda(g)\in A$. A suitable power $g^p$ is contained in the connected component $G^0$ of the identity element, hence the class of $h^n|_{D(h)}$ in $E(D(h))=E(D(h))^{G^0}$ (the equality holds by \cite{kkv} Proposition 1.3) is $g^p$-invariant. We conclude that $\lambda(g)^p\in \G_m$, so $\lambda(g)$ is a constant function on $D(h)$, and hence on $X$, which takes value in $k^*$. Finally, it is easy to see that the assignment $g\mapsto \lambda(g)\in \G_m$ defines a character of $G$.\\
Conversely, $h^n$ and $g*h^n=\chi(g)\cdot h^n$ have the same vanishing order on any prime divisor, and the claim follows by the above computations.\\

Consider the map $F':E(X)^G\times \Z^{r}\to E(U)$ given by
$$(\lambda,m_1,\ldots,m_r)\mapsto \lambda\cdot f_1^{m_1}\cdot\ldots \cdot f_r^{m_r},$$
where the $f_i$ are associated to orbits as in the third assertion. By the second assertion and Lemma \ref{L:invertiblefunctions} this map gives the fourth assertion.
\end{proof}

From now on we will continue using the notation as it is introduced in Lemma \ref{L:actiononZ}. In particular, let $r$ denote the number of orbits of $\Zcal(U)$ under the action of $G$.

\begin{prop} \label{P:picardnumber}
In the situation of Convention \ref{groupsituation} there is an up to torsion exact sequence
$$0\to E(Y)\to E(X)^G\times \Z^{r} \to \chi(G)\to Pic(Y)\to 0,$$
which is that this sequence becomes exact after tensoring with $\Q$. In particular, the formula
$$\rho(Y)=rk(\chi(G))-(|\Zcal/G|+rk(E(X)^G))+rk(E(Y))$$
for the Picard number holds. If the stabilizers for points on $U$ are trivial, $A$ is a UFD, and if furthermore $G$ is connected, then the above sequence is exact by itself.
\end{prop}

\begin{proof}
We want to apply Proposition 5.1 of \cite{kkv} to our situation. For the remainder of the proof we adopt the notation given there, except that we replace $X$ by $U$.\\
Because $U$ is open in $X$ the group
$$Pic(U)\hookrightarrow Cl(U)\hookrightarrow Cl(X)$$
is torsion, and vanishes if $A$ is a UFD. The cokernel of
$$q^*:Pic(Y)\to Pic^G(U)$$
is a subgroup
$$coker(q^*)\subset \prod_{x\in\mathcal{C}}\chi(G_x),$$
where $\mathcal{C}$ is a finite set of points representing closed orbits. To show that this group is finite, it thus suffices to prove that each $\chi(G_x)$ is finite, which is true since the $G_x$ are finite by assumption.\\
Furthermore, because the group $G/G^0$ is finite, the group $H^1(G/G^0,E(U))$ is torsion (consider \cite{weibel} Theorem 6.5.8), and vanishes if $G$ is connected.\\

Up to torsion the diagram thus reduces to an exact sequence
$$0\to E(Y)\to E(U)^G\to \chi(G)\to Pic(Y)\to 0,$$
and the claim is implied by Lemma \ref{L:actiononZ}.
\end{proof}

\begin{defin}
The last map in the exact sequence of Proposition \ref{P:picardnumber} induces a map
$$\psi:\chi(G)_\Q\to Pic(Y)_\Q,$$
given by sending a character to the descent of the trivial line bundle $L$ linearized by that character.\\
By Lemma \ref{L:actiononZ} this map is surjective and the kernel is spanned by the characters $\chi_i\in\chi(G)$ such that $g*f_i=\chi_i(g)\cdot f_i$ for $B_i\in \Zcal/G$ and characters $\chi'\in\chi(G)$ such that $g*f=\chi'\cdot f$ for $f\in E(X)^G$.
\end{defin}

It is easy to see that multiplication with $\chi'$ corresponding to an invertible regular function does not change the semistable locus. This is not true in general for the characters $\chi_i$ associated to $B_i\in\Zcal/G$.

\begin{defin}
A GIT class $C$ is called a stable class with respect to $U$ if $\chi\cdot \chi_1^{m_1}\cdot \ldots \cdot \chi_r^{m_r}$ remains in $C$ for every choice $m_1,\ldots,m_r\geq 0$ and every character $\chi\in C$, where the $\chi_i$ are given as in Lemma \ref{L:actiononZ}.
\end{defin}

The $\chi_i$ as in Lemma \ref{L:actiononZ} can be interpreted as directions in the character group. Because there are only finitely many GIT classes, a ray starting in any character will stay in some fixed class for large distances, and the content of the following lemma is that the rays in the directions of the $\chi_i$ satisfy this simultaneously.

\begin{lemma}\label{L:lifting}
Using the notation of Convention \ref{groupsituation} choose an arbitrary character $\chi\in\chi(G)$. Then after replacing $\chi$ by a suitable multiple there is an isomorphism of section rings
$$R(X,L_{\chi \chi_1^{m_1}\ldots \chi_r^{m_r}})^G \to R(U,L_\chi)^G,$$
given in degree $n$ as
$$s\mapsto \frac{1}{f_1^{nm_1}\cdot\ldots\cdot f_r^{nm_r}}s|_U,$$
if we take the $m_i=m_i(\chi)$ suitably large.\\

In particular, $\chi\cdot \chi_1^{m_1}\cdot \ldots \cdot \chi_r^{m_r}$ is contained in a stable class, for $m_i$ large enough.
\end{lemma}

\begin{proof}
Injectivity is obvious. Consider the map
$$\phi:\bigoplus_{m_1,\ldots,m_r,n\geq 0}H^0(X,L_{\chi^n\chi_1^{m_1}\ldots\chi_r^{m_r}})^G\to \bigoplus_{n\geq 0} H^0(U,L_{\chi^n})^G,$$
given on homogenous elements in a similar form as the map in the statement of the lemma. Note that $\phi$ is graded with respect to the $n$-gradings, and surjective since every section in $H^0(U,L_\chi)^G$ can be lifted to a global section after multiplication with sufficiently many $g_i$.\\
Note that the characters $\chi_i$ are linearly independent by the sequence in Lemma \ref{P:picardnumber} and the isomorphism in Lemma \ref{L:actiononZ}.(4), so that the left hand side is isomorphic to
$$H^0(X,L)[S,T_1,T_2,\ldots,T_r]^G=A[S,T_1,T_2,\ldots,T_r]^G,$$
where the action of $G$ is given as
$$g*S=\chi(g)^{-1}S,~g*T_i=\chi_i(g)^{-1}T_i,$$
and the action on scalars is inherited from the action of $G$ on $A$. This fixed point algebra is finitely generated as a $k$-algebra by the Theorem of Hilbert-Nagata, so the right hand side is finitely generated as well. Of course it is sufficient to lift the finitely many generators, which after taking a suitable thinning of the section ring can be assumed to all live in degree 1. Taking fixed powers $m_1,\ldots,m_r$ which lift all generators thus gives a surjective map as claimed. 
\end{proof}

The following useful observation can be deduced from Lemma \ref{L:lifting}.

\begin{cor}\label{C:locusstable}
If $\chi\in \chi(G)$ is contained in a stable class, then
$$X^{\chi-sst}\cap U=X^{\chi-sst}$$
in codimension one.
\end{cor}

\begin{proof} Clearly, the only codimension one components which could contradict the equality are of the form $Z_i$. But by Lemma \ref{L:lifting} there are isomorphisms
$$H^0(X,L_{\chi \chi_1^{m_1+1}\ldots \chi_r^{m_r+1}})^G \to H^0(U,L_\chi)^G \leftarrow H^0(X,L_{\chi \chi_1^{m_1}\ldots \chi_r^{m_r}})^G$$
for sufficiently large $m_i$, and thus we can assume without loss of generality (replacing $\chi$ by $\chi\cdot\chi_1^{m_1}\cdot\ldots\cdot \chi_r^{m_r}$) that any semiinvariant with respect to $\chi$ can be divided by $f_1,
\ldots,f_r$. So the $Z_i$ are unstable with respect to $\chi$.
\end{proof}

Another consequence of Lemma \ref{L:lifting} is the following.

\begin{prop}\label{L:isosectionrings}
For a character $\chi\in \chi(G)$ contained in a stable class, the canonical map
$$R(X,L_\chi)^G\to R(Y,\psi(\chi))$$
is an isomorphism, where we may have to replace $\chi$ by a suitable multiple.\\

In particular, the pseudoeffective cone $\Eff(Y)$ is exactly the image of the union of all stable $G$-ample classes.
\end{prop}

\begin{proof}
By definition of descent we have an induced isomorphism
$$q^*:R(Y,\psi(\chi))\to R(U,L_\chi)^G.$$
Taking suitable multiples and multiplying with suitable powers of the $\chi_i$ (the latter operation does not change $\psi(\chi)$), the right hand side is isomorphic to $R(X,L_\chi)^G$ by Lemma \ref{L:lifting}.\\

Hence the divisor $\psi(\chi)$ admits a section if and only if $L_\chi$ admits a section, which is equivalent to $X^{\chi-sst}\neq \emptyset$ up to multiples of $\chi$.
\end{proof}

Next we prove some properties of the rational maps induced by GIT (see section \ref{SS:git} for a definition). The proof relies on the fact that $X$ is the spectrum of an AFD $A$, and the author is unsure about the situation in more general cases (compare with the remark about Corollary 1.4 in \cite{thaddeus}). Recall our assumption that stability and semistability with respect to chambers conincide.\\
We remark that the following result holds (with the same proof) in the case where we replace the stable and semistable locus by open sets which admit a geometric and good quotient respectively. Equally well, this holds by an application of Remark \ref{R:allquotientsgitmds}.

\begin{lemma}\label{L:rationalmaps}
Suppose that a reductive group $G$ acts on a variety $X$, which is the spectrum of an AFD $A$. Choose any two characters $\chi,\chi'\in\chi(G)$, where $X^{\chi-st}$ is nonempty, and consider the associated rational map
$$f:Y_\chi\dasharrow Y_{\chi'}$$
induced by GIT. Then the following holds.
\begin{enumerate}
\item The domain of $f$ satisfies
$$q_\chi(X^{\chi-st}\cap X^{\chi'-sst})\subset dom(f)\subset q_\chi(X^{\chi-sst}\cap X^{\chi'-sst}).$$
\item If $\chi$ and $\chi'$ are contained in the interiors of GIT chambers, then the exceptional locus is given as
$$exc(f)=q_\chi(X^{\chi-sst}\cap X^{\chi'-unst}).$$
\end{enumerate}
\end{lemma}

\begin{proof}
In the first assertion the inclusion on the left hand side is obvious.\\
To prove the other inclusion, choose $y\in dom(f)$ and an open neighborhood $W\subset Y_\chi$ such that there is a regular morphism $g:W\to Y_{\chi'}$, extending the morphism on $q_\chi(V)$, where again $V=X^{\chi-st}\cap X^{\chi'-sst}$.\\
Choose any $x\in (q_\chi)^{-1}(W)\subset X^{\chi-sst}$ such that $q_\chi(x)=y$ and $x'\in X^{\chi'-sst}$ such that $q_{\chi'}(x')=g(y)\in Y_{\chi'}$. Up to multiples of $\chi'$ there is a section $s'\in H^0(X,L_{\chi'})^G$ such that $s'(x')\neq 0$.\\
The linearized bundle $(L_{\chi'})|_{X^{\chi'-sst}}$ descends to a semiample line bundle $A$ on $Y_{\chi'}$, and hence $s'|_{X^{\chi'-sst}}=q_{\chi'}^*(\sigma)$ for a section $\sigma\in H^0(Y_{\chi'},A)$. Consider the pullback section
$$s=q_\chi^*g^*(\sigma)\in H^0(q_\chi^{-1}(W),q_\chi^*g^*(A))^G.$$
Since $q_\chi^{-1}(W)\subset X$ is open, the Picard group of $q_\chi^{-1}(W)$ is torsion as well, and hence as a line bundle $q_\chi^*g^*(A)\simeq L$ after replacing $A$ by a multiple if necessary.\\
Note that linearized actions of $G$ on the trivial bundle $L$ are given by multiplication with a global cocycle, so the equality $q_\chi^*g^*(A)=(L_\chi)|_{q_\chi^{-1}(W)}$ of $G$-line bundles can be checked on the open subset $q_\chi^{-1}(W)\cap V$, where it holds using the commutativity of the obvious diagram. Similarly $s$ and $s'$ coincide on $q_\chi^{-1}(W)\cap V$, so $s=s'|_{q_\chi^{-1}(W)}$ can be lifted to a global semiinvariant with respect to $\chi'$. Now we are done since $s(x)=\sigma(g(y))=s'(x')\neq 0$.\\
This proves the first assertion.\\

Using the description of $dom(f)$ given in the first assertion, the second assertion is obvious.
\end{proof}

We now have established the results needed to show that all subsets and quotients of the form as in Convention \ref{groupsituation} are induced by GIT.

\begin{rem}\label{R:allquotientsgit}
Suppose that a reductive group $G$ acts on the spectrum $X$ of an almost factorial domain $A$. Furthermore assume that there exists an open $G$-invariant subset $U\subset X$ and a geometric quotient
$$q:U\to Y,$$
where $Y$ is projective. We then claim that $U=X^{\chi-sst}$ for a character $\chi\in \chi(G)$ in the interior of a GIT chamber.\\

\begin{proof}
We adapt the classical proof of a similar result in the smooth case (\cite{mumford} Converse 1.13) to our situation.\\
Fix an ample line bundle $A\in Pic(Y)$ and consider the pullback
$$q^*(A)\in Pic^G(U).$$
As in the proof of Lemma \ref{L:rationalmaps} we can write $q^*(A)=(L_\chi)|_U$ for a character $\chi\in \chi(G)$ contained in a stable class with respect to $U$, at least after taking a suitable multiple of $A$. Furthermore, by an argument as in Corollary \ref{C:locusstable}, we see that the codimension one components of $X\setminus U$ are unstable with respect to $\chi$. The original proof furthermore employs the fact that the complement of an open affine subset is of pure codimension one, which carries over to the normal case by an application of the Hartogs lemma \ref{L:hartogs}.\\
Now that these facts are established, the proof of \cite{mumford} provides us with the assertion $U\subset X^{\chi-st}$. Hence we have a diagram
$$\xymatrix{U \ar@{^(->}[r] \ar[d]_q & X^{\chi-st} \ar[d] \ar@{^(->}[r] & X^{\chi-sst} \ar[d]^{q_\chi} \\
Y \ar@{^(->}[r] & q_\chi\left(X^{\chi-st}\right) \ar@{^(->}[r] & Y_\chi.}$$
Since $Y$ is projective and $q_\chi\left(X^{\chi-st}\right)$ is irreducible, the open immersion
$$Y\hookrightarrow q_\chi\left(X^{\chi-st}\right)$$
is an isomorphism, and the same reasoning applies to the lower right arrow of the diagram. Because geometric quotients parametrize all orbits, this implies that $U$ and $X^{\chi-st}$ consist of the same orbits. Furthermore, by the properties of a good quotient the closure of any orbit in $X^{\chi-sst}$ must intersect the closure of an orbit in $X^{\chi-st}$, and hence contains it. But orbits in $X^{\chi-st}$ are of maximal dimension, so that these orbits necessarily coincide. In other words, the maps in the upper row are identities.
\end{proof}
\end{rem}

We are finally ready to prove the main result of this section.

\begin{thm}\label{T:quotientmoridream}
In the situation of Convention \ref{groupsituation} the following holds.
\begin{enumerate}
\item Section rings of divisors on $Y$ are finitely generated, the rational maps associated to stable characters are contractions of $Y$, and every rational contraction to a normal, projective variety is of such a form.
\item $Y$ is a Mori dream space.
\item Under the canonical map $\psi:\chi(G)_\Q\to Pic(Y)_\Q$ the image of the $G$-ample stable GIT classes is exactly the pseudoeffective cone of $Y$, and Mori chambers of $Y$ are identified with stable GIT chambers.
\end{enumerate}
\end{thm}

\begin{proof}
It is well-known that a quotient of a normal variety is again normal, and $\Q$-factoriality for $Y$ holds by Lemma 2.1 of \cite{hukeel}. Since we assume that stability and semistability with respect to chambers coincide, the chamber quotients are geometric, and hence normal and $\Q$-factorial by the same argument.\\

All rational contractions are induced by a divisor (\cite{hukeel} Lemma 1.6), and since $\psi$ is surjective, we can write any divisor on $Y$ as $D=\psi(\chi)$ up to multiples (more precisely, the Cartier divisor associated to the line bundle $\psi(\chi)$ is linearly equivalent to $D$). By Lemma \ref{L:lifting} we can further assume that $\chi$ is contained in a stable class, so that the canonical map
$$R(X,L_\chi)^G\to R(Y,\psi(\chi))$$
is an isomorphism by Lemma \ref{L:isosectionrings}. Consider the induced isomorphism
$$\phi:Y_\chi=Proj(R(X,L_\chi)^G)\to Proj(R(Y,D)),$$
and the following diagram
$$\xymatrix{U \ar[d]^q & \ar@{_(->}[l]\ar@{^(->}[r] V \ar[d] & X^{\chi-sst} \ar[d]^{q_\chi} \\
Y \ar@{-->}[rd]_{f_D} & q(V)\ar[r] \ar@{_(->}[l] & Y_\chi\ar[ld]^\phi \\
& Proj(R(Y,D)), & }$$
where again $V=U\cap X^{\chi-sst}$. Note that the two squares at the top form the diagram used in the construction of the rational map $f_\chi:Y\dasharrow Y_\chi$ (see section \ref{SS:git}). Commutativity is clear except for the part involving $f_D$ and the isomorphism $\phi$.\\
If we choose generators $s_0,\ldots s_d$ for $R(X,L_\chi)^G$ and $s'_0,\ldots,s'_d$ for $R(Y,D)$ compatible with $\phi$, and without loss of generality of degree 1, then the GIT quotient $q_\chi$ is given as
$$x\mapsto [s_0(x):\ldots:s_d(x)],$$
and a similar description holds for $f_D$. Hence $f_D\circ q|_V=\phi\circ q_\chi|_V$ and using the surjectivity of the quotient map $V\to q(V)$, this implies that the lower triangle commutes as well. This proves the first assertion.\\

If $\chi$ and $\chi'$ are characters corresponding to the same stable class, then by a similar argument as above $\psi(\chi)$ and $\psi(\chi')$ are Mori equivalent.\\
Since there are only finitely many GIT chambers, which are rational polyhedral, this implies that there are only finitely many Mori chambers, which are rational polyhedral as well.\\

Now, assume conversely that $D$ and $D'$ are general elements of the interior of the same Mori chamber. We can again write $D=\psi(\chi)$ and $D'=\psi(\chi')$ for characters $\chi,\chi'\in\chi(G)$ in stable GIT chambers, considered to be in the interior. We want to prove that $\chi$ and $\chi'$ are GIT equivalent, for which it is sufficient to show $X^{\chi-sst}\subset X^{\chi'-sst}$.\\

By Mori equivalence, there is an isomorphism $\phi:Y_{\chi}\to Y_{\chi'}$, which commutes with the contractions $f_{\chi}$ and $f_{\chi'}$. Using this isomorphism it is easy to see that $exc(f_\chi)$ and $exc(f_{\chi'})$ coincide (or that $dom(f_\chi)$ and $dom(f_{\chi'})$ coincide). Thus, by Lemma \ref{L:rationalmaps} we get
$$U\cap X^{\chi-sst}=U\setminus q^{-1}(exc(f_\chi))=U\cap X^{\chi'-sst}.$$
On the other hand
$$X^{\chi-sst}=U\cap X^{\chi-sst}$$
in codimension one by Corollary \ref{C:locusstable}, and similarly for $\chi'$, and thus the equality
$$X^{\chi-sst}=X^{\chi'-sst}$$
holds in codimension one.\\

Consider the following diagram, where $V=X^{\chi-sst}\cap X^{\chi'-sst}$.
$$\xymatrix{X^{\chi-sst} \ar[dd]_{q_{\chi}} && \ar@{_(->}[ll] \ar@{^(->}[rr] V \ar[ld]_{q_{\chi}} \ar[rd]^{q_{\chi'}} && X^{\chi'-sst} \ar[dd]^{q_{\chi'}} \\
 & q_{\chi}(V) \ar[rr]^{\sim} \ar@{_(->}[ld] && q_{\chi'}(V) \ar@{^(->}[rd] &  \\
Y_{\chi} \ar[rrrr]^{\stackrel{\phi}{\sim}}  &&&& Y_{\chi'} \\
&q_{\chi}(U\cap V) \ar[rr]_{\sim} \ar@{_(->}[lu] && q_{\chi'}(U\cap V) \ar@{^(->}[ru] & \\
&&q(U\cap V) \ar[ul]_{\sim} \ar[ur]^{\sim} \ar@{^(->}[d] && \\
&&Y \ar@{-->}@/^13mm/[uuull]^{f_{\chi}} \ar@{-->}@/_13mm/[uuurr]_{f_{\chi'}} &&}$$

Note that $q(U\cap V)=q(U\cap X^{\chi-sst})\cap q(U\cap X^{\chi'-sst})$ is an open subset, where both contractions $f_{\chi},f_{\chi'}$ are defined. For the commutativity of the square involving $\phi$ and the isomorphism $q_{\chi}(V)\simeq q_{\chi'}(V)$, note that both maps coincide on the open subset $q_{\chi}(U\cap V)$. The commutativity of the other squares is immediately clear by construction.\\

Pick any point $x\in X^{\chi-sst}$. \\
By the GIT-construction we have the following for any point $y\in Y_{\chi'}$: there exists a section $s\in H^0(X,L_{\chi'})^G$ such that $s|_{X^{\chi'-sst}}=q_{\chi'}^*(s')$, where $s'\in H^0(Y_{\chi'},A')$ is a section of the ample line bundle $A'$ on $Y_{\chi'}$ given as the descent of the bundle $(L_{\chi'})|_{X^{\chi'-sst}}$, such that $s'(y)\neq 0$. In particular, this holds for $y=(\phi\circ q_{\chi})(x)\in Y_{\chi'}$.\\

Via $\phi$ the ample line bundle $A'$ can be identified with an ample line bundle $A$ on $Y_{\chi}$, and $s'$ with a section $s''$, and working through the diagram it is easy to check that
$$(L_{\chi'})|_{V}\simeq q_{\chi}^*(A|_{q_{\chi}(V)}).$$
A similar statement holds for $s$ and $s''$.

We thus have a diagram
$$\xymatrix{H^0(X^{\chi-sst},L_{\chi'})^G \ar@{^(->}[r] & H^0(V,L_{\chi'})^G \ar[d]^\sim \\
H^0(X^{\chi-sst},q_{\chi}^*(A))^G \ar@{^(->}[r] & H^0(V,q_{\chi}^*(A))^G.}$$

By the computation above $X^{\chi-sst}=X^{\chi'-sst}$ in codimension one, so by the Hartogs lemma \ref{L:hartogs} the map in the first line is an isomorphism.\\
This finally gives an identification
$$s|_{X^{\chi-sst}}=q_{\chi}^*(s'')\in H^0(X^{\chi-sst},L_{\chi'})^G,$$
and thus $s(x)=s''(q_{\chi}(x))=s'(\phi\circ q_{\chi}(x))\neq 0$, so $x$ is semistable with respect to $\chi'$ as desired. This proves the third assertion.\\

For the second assertion it remains to check that the Mori chambers constituting the moving cone are spanned by divisors which are pullbacks of semiample divisors under SQM's $f_i:Y\dasharrow Y_i$, and that the nef cones of the $Y_i$ are spanned by finitely many semiample bundles (note that $id:Y\to Y$ is a SQM as well).\\

We first observe that the pullback $f^*(D)$ of a semiample divisor $D$ under a SQM $f:Y\dasharrow Y'$ is movable. Indeed, using the open subsets $V\subset X$ and $W\subset Y'$ such that $f:V\to W$ is an isomorphism, and that the complements of $V$ and $W$ have codimension at least two, $f^*(D)$ is up to multiples the extension of the usual pullback of $D|_W$ to $V$ (consider section 1.2 of \cite{casagrande}).\\
Thus by the Hartogs lemma
$$H^0(Y,f^*(D))\simeq H^0(V,f^*(D))\simeq H^0(W,D)\simeq H^0(Y',D),$$
so $V$ is not contained in the stable base locus.\\

Conversely, the contraction $f_D$ associated to a movable divisor $D$ contained in the interior of a Mori chamber on $Y$ is Mori equivalent to the GIT contraction associated to a character $\chi$ contained in the interior of a stable GIT chamber. By Lemma \ref{L:rationalmaps}, the map $f_D$ is an isomorphism on its domain, which is of codimension at least two, and hence a SQM. By Lemma 1.6 in \cite{hukeel} it follows that $D$ is the pullback of an ample divisor under the SQM $f_D$.\\

This proves that $\Mov(Y)$ is the union of the cones $f_i^*(\Nef(Y_i))$, and we only need to establish that the cones $\Nef(Y_i)$ are spanned by semiample bundles.\\

Keeping the notation
$$f_D\sim f_\chi:Y\dasharrow Y_\chi$$
for $D$ contained in the interior of a Mori chamber $\C\subset \Mov(Y)$, and $\chi$ contained in the interior of a stable GIT chamber, we claim that pullback by $f_D$ identifies $\Nef(Y_\chi)$ with $\C$. To see this we need to establish some properties of the pullback.\\
Since the compositions of two Mori equivalent rational maps with a third rational map are again Mori equivalent, we see that pullback respects Mori chambers. On the other hand, $f_\chi^*$ and $(f_\chi^{-1})^*$ are inverse to each other, because $(f_\chi^{-1})^*$ is again a SQM and pullback is functorial with respect to the composition of SQMs (see Remark 2.8 in \cite{casagrande}).\\
With these preparations our claim about $\C$ and $\Nef(Y_\chi)$ follows, since the pullback of the ample cone is contained in $\C$ by Lemma 1.6 in \cite{hukeel}.\\

A nef divisor $D''$ on $Y_\chi$ thus gets identified with a divisor $D'=(f_\chi)^*(D'')$ contained on the boundary of $\C$. This divisor can be written as $D'=\psi(\chi')$, where $\chi'\in\chi(G)$ is contained in a stable GIT class contained in the closure of the stable GIT chamber in which $\chi$ is contained. This implies that $X^{\chi-sst}$ is a subset of $X^{\chi'-sst}$ (use a linearization to reduce to the case of a $G$-module, and apply \cite{halic} Lemma 3.6), and thus
$$f_{\chi'}\circ (f_\chi)^{-1}: Y_\chi \dasharrow Y_{\chi'}$$
can be considered as a regular map by Lemma \ref{L:rationalmaps}. The pullback
$$\overline{D}=(f_{\chi'}\circ (f_\chi)^{-1})^*(\Ocal(1))$$
in the usual sense, that is with respect to a regular morphism, is semiample on $Y_\chi$. On the other hand, $f_{\chi'}^*(\Ocal(1))=D'+E$ for an $f_{\chi'}$-exceptional divisor $E$ by Lemma 1.6 in \cite{hukeel}, and thus
$$D''|_V=(f_{\chi'}\circ (f_\chi)^{-1})^*(\Ocal(1))|_V,$$
where $V\subset Y_\chi$ is an open subset such that $f_{\chi}^{-1}$ restricted to $V$ is an isomorphism onto its image, and such that $f_{\chi'}$ is regular on $f_\chi^{-1}(V)$. Observe that $\overline{D}$ agrees with $D''$ after restriction to $V$, and further that, by being semiample, it can be moved away from the exceptional set, and thus agrees with $D''$ globally.
\end{proof}

\begin{rem}
In the situation of Convention \ref{groupsituation}, choose any $G$-ample character $\chi$. In any case we can find a GIT chamber $C$ such that $\chi$ is contained in its closure.\\
The quotient $Y$ associated to the chamber is a Mori dream space by Theorem $\ref{T:quotientmoridream}$, and still $Y_\chi$ is normal and projective but may fail to be $\Q$-factorial. There is a regular contraction
$$Y\to Y_\chi$$
by Lemma \ref{L:rationalmaps}, so $Y_\chi$ is at least a not necessarily $\Q$-factorial Mori dream space by the result of \cite{okawa}, Section 10.1.\\
Apart from an upper bound on the Picard number, and the assertion that the Mori chamber structure of $Y_\chi$ is a coarsening of the Mori chamber structure of $Y$, this observation does not seem to provide any quantitative results, and the qualitative result is already covered by \cite{baeker}. Even though this does not seem to be helpful in itself, it sparks the hope that the results of the present document might possibly extend to good quotients.
\end{rem}

\section{Application to quiver moduli \label{S:quiver}}
Quiver moduli were introduced to study the isomorphism classes of modules over artinian algebras (see \cite{king}). Here we do not pursue this approach, and rather focus on the variety structure of these spaces. From our point of view, they form a class of examples for Mori dream spaces, which admits a description in a combinatorial flavor. An explicit instance of this philosophy is given in Examples \ref{ex:quiver} and \ref{ex:hirzebruch}.\\
Conversely, the description of quiver moduli as Mori dream spaces might be helpful in understanding their birational geometry.\\
It has been observed before that Theorem 2.3 in \cite{hukeel} can be applied to some quiver moduli (see for example \cite{craw}), but the assumption on the codimension of the unstable locus seems to restrict the possible applications.\\

Let us first explain how quiver moduli are constructed. For that we loosely follow \cite{reineke}, where the reader may find further details and applications.\\

A quiver $Q=(Q_0,Q_1)$ consists of a set of vertices $Q_0$, and a set of arrows $\alpha:i\to j$ between vertices, where we assume both sets to be finite. A dimension vector is a tuple $d\in\N^{Q_0}$. The representation variety with respect to $Q$ and $d$ is defined as
$$R_d(Q)=\bigoplus_{\alpha:i\to j} \Mat(d_j\times d_i,k).$$
Note that this is just another way to write down affine space.\\

There is a canonical action of the reductive group
$$G_d=\prod_{i\in Q_0}\GL_{d_i}(k)$$
on $R_d(Q)$, given as simultaneous conjugation
$$(g_i)_{i\in Q_0}*(f_\alpha)_{\alpha:i\to j}=(g_j\cdot f_\alpha\cdot g_i^{-1})_{\alpha:i\to j}.$$
The diagonally embedded scalars $\G_m\hookrightarrow G_d$ act trivially on $R_d(Q)$, which induces an action of $PG_d=G_d/\G_m$ on $R_d(Q)$. It is well-known that there is an isomorphism
$$\Z^{Q_0}\to \chi(G_d),$$
defined as
$$\theta=(\theta_i)_{i\in Q_0}\mapsto \left((g_i)_{i\in Q_0}\mapsto \prod_{i\in Q_0} \det(g_i)^{-\theta_i}\right).$$
The $\theta\in \Z^{Q_0}$ are sometimes referred to as stability conditions, and it is easy to see that the characters of $PG_d$ are identified with the subgroup $d^\perp\subset \Z^{Q_0}$ (of the $\theta$ such that $\theta\cdot d=\sum_{i\in Q_0} \theta_i\cdot d_i=0$).\\

For any given quiver $Q$, dimension vector $d$, and stability condition $\theta$, the associated quiver moduli $\M_{d,\theta}(Q)$ is defined as the good quotient of $R_d(Q)^{\theta-sst}$ under the action of $PG_d$. It is well-known that $\M_{d,\theta}(Q)$ is projective if $Q$ does not admit oriented cycles.\\

Obviously $PG_d$ is connected, and $R_d(Q)$ is the spectrum of a polynomial ring (and hence of a UFD), so we obtain the following corollary.\\
Note that if $d$ is coprime, which is that the greatest common divisor of the $d_i$ equals 1, the set of stability conditions for which stability and semistability may differ is a finite union of hyperplanes (see Section 3.5 in \cite{reineke}). Thus for GIT chambers stability and semistability coincide.

\begin{thm}\label{T:quivermds}
Suppose that $Q$ is a quiver which does not admit oriented cycles, that $d$ is a dimension vector, and that $\theta$ is a stability condition contained in the interior of a GIT chamber.\\
Then $\M_{d,\theta}(Q)$ is a Mori dream space satisfying the assertions of Theorem \ref{T:quotientmoridream}. In particular, the Picard number is given as
$$\rho(\M_{d,\theta}(Q))=|Q_0|-(r+1),$$
where $r$ is the number of components of the unstable locus which are of codimension one.
\end{thm}

If $Q$ admits oriented cycles, $\M_{d,\theta}(Q)$ is no longer projective, and hence fails to be a Mori dream space for trivial reasons. Since $\M_{d,\theta}(Q)$ is projective over the Hilbert quotient, the author expects that this however gives an example of a relative Mori dream space, which as of yet needs to be defined.\\

Furthermore, we note that the same results hold if we allow relations for the quiver, such that the associated representation variety is again a spectrum of an AFD. However, it seems very hard to describe relations which meet this criterion in general.

\begin{rem}\label{R:quivermdsalternative}
There is an alternative proof showing that quiver moduli as in Theorem \ref{T:quivermds} are Mori dream spaces, which only makes use of the quotient result established in \cite{hukeel}.\\
Indeed, consider the framed quiver $\widehat{Q}$ with additional stability parameter $\eps$ attached to the additional point $\infty$ (we refer to \cite{engelreineke} Section 3 for more details). We may take our framing to be large, in the sense that for each vertex $i\in Q_0$ we attach at least $n_i=d_i+1$ arrows $\infty\to i$.\\
Similarly to \cite{engelreineke} Proposition 3.3, we see that for $\eps\gg 0$ a representation $(M,f)$ of the framed quiver is stable if and only if it is semistable if and only if $f$ is dense, which is that there is no proper subrepresentation containing the vector space $k$ attached to $\infty$. Since the framing is large, we can always equip the arrows landing in each vertex $i\in Q_0$ with linear maps such that their images generate the vector space at $i$, and thus the unstable locus is contained in
$$R_{d}(Q)\times \prod_{i\in Q_0}Mat(d_i\times (d_i+1))_{rk\leq d_i-1},$$
which is well-known to be of codimension at least two. By Theorem 2.3 in \cite{hukeel} the associated quotient $\widehat{\M}_{\eps\gg 0}$ is a Mori dream space, and so is the other chamber quotient $\widehat{\M}_{\eps\approx 0}$, where we take $\eps$ to be positive but sufficiently small (\cite{hukeel} Proposition 1.11.(2)). However, for this chamber there is the structure of a bundle
$$\widehat{M}_{\eps\approx 0}\to \M,$$
where $\M$ is the moduli of the unframed quiver (\cite{engelreineke} Proposition 3.8). Hence $\M$ is a Mori dream space as well by the main result of \cite{okawa}.\\
However, it seems difficult to extract information about the Picard number or the Mori chambers using this construction.\\
Of course we could directly apply the result of \cite{baeker}, since the prequotient $R_d(Q)$ is just affine space, but this does provide us with no quantitative information about the moduli space at all.
\end{rem}

\section{Quotients of Mori dream spaces \label{S:quotientofmds}}
Let $Y$ denote a Mori dream space, acted upon by a reductive group $G$. Again we assume that stability and semistability with respect to chambers coincide (compare with section \ref{SS:git}). Suppose that there is a $G$-invariant open subset $V\subset Y$ admitting a geometric quotient
$$q':V\to Z,$$
where $Z$ is projective and stabilizers on $V$ are finite. We want to show in this section that $Z$ is again a Mori dream space, and compute the Picard number.\\
Again, by Remark \ref{R:allquotientsgitmds} given below, such a quotient is given as a GIT quotient with respect to the interior of a chamber. But the proofs occurring in this section do not depend on this observation.\\

Let $r=\rho(Y)$ denote the Picard number of $Y$, and choose line bundles
$$D_1,\ldots,D_r\in Pic(Y),$$
which form a $\Q$-basis of $Pic(Y)_\Q$, and such that $\Eff(Y)$ is contained in the cone spanned by these bundles. Then the Cox ring
$$Cox(Y)=\bigoplus_{m_1,\ldots,m_r\geq 0}H^0\left(Y,D_1^{m_1}+\ldots+D_r^{m_r}\right)$$
with respect to this choice of a basis is a finitely generated $k$-algebra, and there is a quotient representation
$$q:U\to Y,$$
where $U\subset X=Spec(Cox(Y))$ is the stable and semistable locus for some character $\chi_0\in \chi(T)$, and $T=Hom(\Z^r,\G_m)\simeq \G_m^r$ is a torus acting on $Cox(Y)$, and hence on $X$ ([HK] Proposition 2.9). The codimension of $X\setminus U$ is greater or equal to two.\\

\begin{rem}\label{R:coxnounits}
It is easy to see that
$$E(X)=Cox(Y)^*/k^*=0$$
using the projectivity of $Y$ and the multigrading of $Cox(Y)$ (compare with Corollary 2.2 in \cite{arzhantsev}).
\end{rem}

A lifting of the action $\mu_G:G\times Y\to Y$ is defined to be an action
$$\mu_G':G\times X\to X$$
such that $U$ is $G$-invariant, and such that the actions of $T$ and $G$ on $X$ commute. Obviously, this is equivalent to an action of the product group $T\times G$ on $X$.\\

As a first step we ensure the existence of a lifting to a suitably chosen Cox ring. Additionally we establish that we may take the Cox ring to be an AFD. Note that if $Pic(Y)$ is torsion-free we may even assume $Cox(Y)$ to be a UFD (consider \cite{arzhantsev}).
\begin{lemma}\label{L:coxlifting}
There exists a basis $D_1,\ldots,D_r$ of $Pic(Y)_\Q$ such that the following assertions hold.
\begin{enumerate}
\item The Cox ring with respect to this basis is almost factorial.
\item There exists a lifting of the action $\mu_G:G\times Y\to Y$ to $X$.
\end{enumerate}
\end{lemma}

\begin{proof}
We start with any basis $D_1,\ldots,D_r\in Pic(Y)$ of $Pic(Y)_\Q$ such that $\Eff(Y)$ is contained in the cone spanned by these bundles.\\
The proof of [HK] Proposition 2.9 shows that $T$ acts freely on $U$ if we replace the $D_i$ with sufficiently high powers. In the terminology of Jow (see [J] Theorem 1.8) such a basis is called a preferred basis, and the associated Cox ring is normal ([J] Proposition 1.12).\\
We again consider the commutative diagram given in [KKV] Proposition 5.1, where $X$ in the notation of [KKV] corresponds to $U$, $G$ to $T$, and $X//G$ to $Y$. Using the normality of $U$ we add the extension ([KKV] Lemma 2.2)
$$Pic(U)\to Pic(T).$$
Using the Hartogs lemma \ref{L:hartogs} the group $E(U)\simeq E(X)$ vanishes by the Remark \ref{R:coxnounits} above. Further $H^1(T/T_0,E(U))$ vanishes since $T$ is connected, and $\prod_{x\in\C} \chi(T_x)$ vanishes because the action of $T$ on $U$ is free. Finally, $T$ is isomorphic to an open subset of affine space $\A^r$, so $Pic(T)=0$. The diagram thus induces an exact sequence
$$0\to \chi(T)\to Pic(Y) \to Pic(U)\to 0,$$
where $\chi(T)$ and $Pic(Y)$ are both of rank $r$. Since $U$ is $\Q$-factorial ([HK] Lemma 2.1), this implies the first result.\\

To see the second assertion we fix a $G$-linearization for each bundle $D_i$, which is possible since $Y$ is normal. These linearizations induce a linearization of each product
$$D^a=D_1^{a_1}+\ldots+D_r^{a_r}$$
for $a\in\N^r$. Note that the linearization of $D^a$ is uniquely determined by the linearizations of the $D_i$ since the decomposition of $D^a$ as a product of the $D_i$ is unique. This gives an action of $G$ on $H^0(Y,D^a)$, and thus on $Cox(Y)$ and $X$ as well.\\

Recall that the action of $T=Hom(\Z^r,\G_m)$ on $Cox(Y)$ is given as
$$t*s=t(a)\cdot s$$
for $t\in T$ and $s\in H^0(Y,D^a)$ homogenous of multidegree $a\in\N^r$. Since the action of $G$ is given by a linear action on the fibres of $D^a$ these two actions commute.\\

The commuting actions of $G$ and $T$ on $X$ induce commuting contragredient actions on $H^0(X,L)=\Ocal(X)$. Hence, if $f\in H^0(X,L_\chi)^T$ is a $T$-eigenfunction with respect to some character $\chi\in\chi(T)$, which does not vanish in a point $x\in X$, then $g*f$ for $g\in G$ is again a $T$-eigenfunction with respect to $\chi$, which does not vanish in the point $g*x$.\\

This implies that the semistable locus with respect to any character $\chi$ of $T$ is $G$-invariant; in particular this holds for $U$.
\end{proof}~\\

Fixing a basis which satisfies the assertions of the lemma above, we obtain the following result. The proof is straightforward.

\begin{lemma}\label{L:quotientcomposition}
Let $V\subset Y$ denote a $G$-invariant open subset such that there exists a good quotient
$$q':V\to Z.$$
Then the composition
$$q^{-1}(V)\xrightarrow{q}V\xrightarrow{q'}Z$$
is a good quotient with respect to the action of $T\times G$ on $X$. If $q'$ is geometric, so is the composition.
\end{lemma}

\begin{rem}\label{R:allquotientsgitmds}
A similar assertion as in Remark \ref{R:allquotientsgit} holds if we replace $X$ by a Mori dream space $Y$.

\begin{proof}
Replacing the notation $q:U\to Y$ by a geometric quotient
$$q':V\to Z,$$
where $Z$ is projective and $V$ is an open and $G$-invariant subset in a Mori dream space $Y$, we get that $q^{-1}(V)\to Z$, where $q$ is given as in Lemma \ref{L:quotientcomposition} and the preceding construction, is a geometric quotient as in the case $X=Spec(A)$ treated in Remark \ref{R:allquotientsgit}. Using the descent properties of the quotient $q:U\to Y$, with notation again as in Lemma \ref{L:quotientcomposition}, we can thus lift the pullback $(q')^*(A)\in Pic^G(V)$ of an ample line bundle on $Z$ to a $G$-line bundle on $Y$, and the proof of Remark \ref{R:allquotientsgit} applies.
\end{proof}
\end{rem}

We are now able to prove the main result of this section.

\begin{thm}\label{T:quotientofmds}
In the situation above, $Z$ is a Mori dream space with Picard number
$$\rho(Z)=\rho(Y)+rk(\chi(G))-|\Zcal(V)/G|.$$
If $G$ is connected, $Pic(Y)$ is torsionfree, and stabilizers of points in $V$ are trivial, the group $Pic(Z)$ is torsionfree as well. Furthermore, under the canonical maps
$$Pic(Y)_\Q \hookrightarrow \chi(T\times G)_\Q \twoheadrightarrow Pic(Z)_\Q,$$
the intersections of Mori chambers of $Z$ with the image of $Pic(Y)_\Q$ is a refinement of the subsystem of Mori chambers of $Y$, consisting of the chambers whose images in $\chi(T\times G)_\Q$ lie in stable GIT chambers with respect to the action of $T\times G$ and the open subset $q^{-1}(V)$.
\end{thm}

\begin{proof}
Using the Lemmata \ref{L:coxlifting} and \ref{L:quotientcomposition} we can write $Z$ as a geometric quotient of the open set $U'=q^{-1}(V)$. By the codimension assertion on the unstable locus with respect to $\chi_0$, the codimension one components of $X\setminus U'$ in $X$ are in one-to-one correspondence under $q$ with the codimension one components of $Y\setminus V$ under $q$. Furthermore, since $T$ is connected, the orbits of the action of $G$ on $\Zcal(V)$ are identified with the orbits of $T\times G$ on $\Zcal(U')$.\\

Hence, $Z$ is a Mori dream space by Theorem \ref{T:quotientmoridream}, with Picard number
$$\rho(Z)=rk(\chi(T\times G))-|\Zcal(V)/G|.$$
The claimed formula follows using
\begin{align*}
\chi(T\times G)&=\chi(T)\times \chi(G),\\
rk(\chi(T))&=rk(T)=\rho(Y).
\end{align*}
Under the additional assumptions we see that $T\times G$ is connected, $X$ is the spectrum of a UFD, and that stabilizers for the action of $T\times G$ on $q^{-1}(V)$ are trivial, whence the torsionfreeness of $Pic(Z)$.\\

Describing the chambers is possible by applying the functor $Pic^\bullet(\bullet)\otimes \Q$ (see Lemma \ref{L:picfunctor}) to the diagram
$$\xymatrix{(Y,*) & \ar[l]_q (U,T) \ar@{^(->}[r] & (X,T)  & \\
&&& (X,G\times T) \ar[lu]_{(id,pr_2)} \\
 (Z,*) && (q^{-1}(V),G\times T), \ar[ll]_{q'\circ q} \ar@{^(->}[ur] & }$$
to obtain
$$Pic(Y)_\Q\simeq \chi(T)_\Q\hookrightarrow \chi(G\times T)_\Q \twoheadrightarrow Pic(Z)_\Q.$$
By the proof of Theorem \ref{T:quotientmoridream} we know that Mori chambers of $Z$ are identified with GIT chambers in $\chi(T\times G)_\Q$, and similarly Mori chambers of $Y$ correspond to GIT chambers in $\chi(T)_\Q$. These GIT chamber structures are compatible by an application of the Hilbert-Mumford criterion (consider section \ref{SS:git}).
\end{proof}

\section{Examples}

We start with a very basic example, showing that Theorem 2.3 of \cite{hukeel} can be extended to the case of unstable components in codimension one.

\begin{example}
Let the group $G=\G_m\times \G_m$ act on the variety
$$X=\A^1\times \A^2$$
via $(g,h)*(x,y)=(gx,hy)$, where we denote by $x$ the coordinate of $\A^1$ and by $y_1,y_2$ the coordinates on $\A^2$. It is easy to verify that there is only one GIT chamber with stable locus
$$U=\left(\A^1\setminus 0\right)\times \left(\A^2\setminus 0\right),$$
and associated quotient
$$q:U\to \P^0\times \P^1\simeq \P^1.$$
The quotient is a Mori dream space despite the fact that the unstable locus  is of codimension one. However, the Picard number $\rho(\P^1)=1$ differs from the rank of the character group $rk(\chi(G))=2$, as is predicted by Theorem \ref{T:quotientmoridream}. Indeed, we have that
$$\Zcal(U)=\left\{0\times \A^2\right\}$$
consists of a single component. Furthermore the relation between the stable GIT chambers and the Mori chambers is correctly predicted as well (though both are trivial). To see this note
$$0\times \A^2=N(x),~(g,h)*x=g^{-1} x,$$
so if we identify $\varphi:\Z^2\simeq \chi(G)$ via
$$\varphi(a,b)=\left((g,h)\mapsto g^{-a} h^{-b}\right)$$
the character associated to $0\times \A^2$ is given as the vector $(-1,0)$ and the only GIT chamber is given as the cone spanned by the rays through the origin and $(-1,0)$ or $(0,-1)$ respectively.
\end{example}

In the first example we can easily replace $\A^2$ by any other prequotient.

\begin{example}
Suppose that a reductive group $G$ acts on $X=Spec(A)$, where $A$ is an AFD. Consider the action of the group $G'=G\times \G_m$ on
$$X'=X\times \A^1$$
given by $(g,h)*(x,y)=(g*x,h y)$. For the character lattice this corresponds to the addition of an orthogonal direction
$$\chi(G')\simeq \chi(G)\times \Z.$$
Since the action is diagonal, the Hilbert-Mumford criterion implies
$$\left(X'\right)^{\chi-sst}=X^{\chi_1-sst}\times \left(\A^1\right)^{\chi_2-sst},$$
where $\chi=(\chi_1,\chi_2)$ is the decomposition with respect to the above isomorphism. A similar statement holds for the stable loci. Hence the chamber structure for the action of $G$ on $X$ and for the action of $G'$ on $X'$ coincide in the sense that the chambers are stretched into the additional direction.\\
One can compute that for $U'$ equal to the semistable locus to some chamber in $\chi(G')_\Q$ the unstable codimension one components are given as
$$\Zcal(U')=\left\{Z\times \A^1~|~Z\in\Zcal(U)\right\}\cup\left\{X\times 0\right\},$$
where $U$ is the semistable locus corresponding to the chamber in $\chi(G)_\Q$. The action of $G'$ on $\Zcal(U')$ is induced by the action of $G$ on $\Zcal(U)$, where $X\times 0$ is fixed.\\
The associated characters for $Z\times \A^1$ lie in the hyperplane $\chi(G)_\Q$, and the character associated to $X\times 0$ points into the orthogonal direction. Thus, the stable chamber structure in $\chi(G')$ is induced by the stable chamber structure in $\chi(G)$, which agrees with the description of the quotients
$$Y'\simeq Y\times \P^0\simeq Y,$$
where $Y'$ is the quotient with respect to a chamber in $\chi(G')_\Q$, and $Y$ is the quotient with respect to the corresponding chamber in $\chi(G)_\Q$.
\end{example}

Our first nontrivial example is the Hirzebruch surface $\F_1$, which also allows a description as a quiver moduli.

\begin{example}\label{ex:hirzebruch}
Consider the vector bundle $E=\Ocal(0)\oplus\Ocal(-1)\to \P^1$, and the associated projectivization
$$\F_1=\P(\Ocal(0)\oplus \Ocal(-1))\to \P^1,$$
which is the first Hirzebruch surface $\F_1$. Denoting by $\pi:\A^2\setminus 0\to \P^1$ the canonical quotient map, we have the pullback $\pi^*E=L_0\oplus L_{-1}\to \A^2\setminus 0$, where as always $L\to \A^2\setminus 0$ is the trivial bundle, and $0$ and $-1$ correspond to characters of $\G_m$ via the isomorphism $\chi(\G_m)\simeq \Z$. Note that as a variety $\pi^*E\simeq \A^2\times (\A^2\setminus 0)$, so taking out the image of the zero section gives the variety
$$(\pi^*E)_0\simeq (\A^2\setminus 0)\times (\A^2\setminus 0).$$
There are two natural actions of $\G_m$ on $(\pi^*E)_0$. The first one is given by the description $\pi^*E=L_0\oplus L_{-1}$ and the second action is induced by the action of $\G_m$ on the fibres of $E$ (corresponding to the quotient $\pi':E_0\to \P(E)$). Since these two actions commute, we obtain an action of $G=\G_m\times \G_m$ on $(\pi^*E)_0$, which reads as
$$(\lambda,\mu)*((a,b),(c,d))=((\lambda a,\lambda b),(\mu c,\lambda \mu d)).$$
We claim that the composition $(\pi^*E)_0\to E_0 \to \P(E)=\F_1$ is a geometric quotient. Indeed, the construction of the involved bundles works locally over the open subsets where $E$ is trivial or the inverse images under $\pi$ thereof, and on these sets the claim holds.\\
We note that $(\pi^*E)_0\subset \A^4$ is the stable and semistable locus associated to the character $\chi(\lambda,\mu)=\lambda^2\mu$. Computing the GIT chambers for this action yields the following result, which is both accessible by an elementary calculation or by a quiver-like computation as in Example \ref{ex:quiver} (for the involved quiver see below).\\

\begin{center}
\begin{tabular}{|c|c|c|}
\hline 
 & stable locus & quotient \\ 
\hline 
\hline
I & $(a,b)\neq 0 \wedge(c\neq 0\vee d\neq 0)$ & $\F_1$ \\ 
\hline 
II & $c\neq 0 \wedge ((a,b)\neq 0\vee d\neq 0)$ & $\P^2$ \\ 
\hline 
\end{tabular} 
\end{center}~\\

For chamber II, the formula for the Picard number as in Theorem \ref{T:quotientmoridream} thus gives the correct result. We further remark that the only stable chamber with respect to chamber II is again chamber II, which agrees with the Mori chamber structure of $\P^2$.\\
Also, for chamber I, the theorem gives the correct result since there are no unstable codimension one components, and the GIT chamber structure agrees with the known Mori chamber structure of $\F_1$.\\

Note that the action of $\G_m\times \G_m$ on $\A^4$ as above admits a description as the action of $PG_d$ on $R_d(Q)$, where the quiver $Q$ is given as
$$
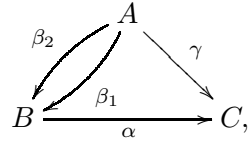
\xymatrix{& A \ar@/_/[ld]_{\beta_2} \ar@/^/[ld]^{\beta_1} \ar[rd]^{\gamma} & \\
B \ar[rr]_{\alpha} && C,}$$
and the dimension vector as $d=(1,1,1)$ (compare with Example 3.6 in \cite{crawsmith}).
\end{example}

The following example gives nontrivial examples of the stable chamber structure. Note that the group $G_d$ associated to the occurring dimension vector is a torus, and that the following quotients also admit a description in the toric language.

\begin{example}\label{ex:quiver}
Let $Q$ denote the following quiver.
$$
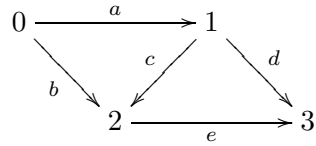
\xymatrix{0 \ar[rr]^a \ar[dr]_b && 1 \ar[ld]_c \ar[rd]^d & \\
& 2 \ar[rr]_e && 3}$$
In the following we will always consider the dimension vector $d=(1,1,1,1)$. If $M\in R_d(Q)$ is a representation we use the notation $M=(a,b,c,d,e)$, where $a\in k$ is the linear map associated to the arrow $a$ and so forth.\\

As explained in section \ref{S:quiver} the character group of $PG_d$ is identified with the hyperplane $(1,1,1,1)^\perp\subset \Z^4\simeq \chi(G_d)$. To facilitate the discussion we will always extend the scalars to $\R$, and work with coordinates on $\chi(PG_d)_\R$ given by the basis
$$\mathcal{B}:~b_1=(1,1,-1,-1),~b_2=(1,-1,1,-1),~ b_3=(1,-1,-1,1)$$
for this hyperplane. To further simplify the picture, we consider the intersection of the GIT cone in $\R^3$ with an affine hyperplane $S$, given in parameter form as
$$S:~x(s,t)=p+su_1+tu_2,$$
where $p=(1,1,0), u_1=(1,-1,0)$, and $u_2=(0,0,1)$. We can interpret this as the hyperplane of points which have value $1$ in direction $(1,1,0)$. A GIT chamber in $\chi(PG_d)$ is thus represented by a polytope in the plane $\R^2$ corresponding to the coordinates on $S$.\\

We need to translate between cones, or half spaces, in $\chi(PG_d)$, and half spaces in $S$. This is a standard linear algebra computation, which gives the following result.\\

Let $v\in(1,1,1,1)^\perp\subset \R^4$ denote a nonzero vector. Then the restriction of $v^\perp\subset \R^4$ to the slice $S$ in the coordinates as above is given as
$$x(t)=a+tb,$$
where
$$b=(v_1+v_4,v_3-v_2),$$
and $a$ is given by
$$a=\left(0,\frac{v_4-v_1}{v_1+v_4}\right),~\mathrm{or}~a=\left(\frac{v_1-v_4}{v_3-v_2},0\right).$$
If $v_2-v_3>0$, then the restriction of the half space $H(v)\subset \R^4$ with the slice $S$ is given by translating the line given above in the direction of the first coordinate. If $v_2-v_3<0$, the same holds with a change of direction.\\
For the translation in the direction of the second coordinate a similar assertion holds when we use the inequalities $v_1+v_4>0$, or $v_1+v_4<0$.\\

We are now ready to compute the intersections of the 14 potential walls with $S$. Instead of a complete computation we focus on two prototypical examples and state the full result in Appendix \ref{app:walls}.\\

First consider the dimension vector $e=(0,0,1,1)$. This dimension vector does always occur as the dimension type of a subrepresentation, and hence the associated hyperplane supports the $G$-ample cone. The condition for slope semistability (consider Theorem 3.8 in \cite{reineke}) reads as
\begin{align*}
\frac{\theta_2+\theta_3}{2}\leq \frac{\theta_0+\theta_1+\theta_2+\theta_3}{4}.
\end{align*}
Since this is equivalent to $0\leq \theta_0+\theta_1-\theta_2-\theta_3$, the associated half space in $\chi(PG_d)_\R$ is given as $H(1,1,-1,-1)$. A necessary condition for a character to be $G$-ample is thus that it is contained in the half space $H(1,1,-1,-1)$.\\

The dimension vector $e=(1,1,0,0)$ occurs as a subrepresentation of a representation $M=(a,b,c,d,e)\in R_d(Q)$ if and only if the maps starting in one of the points $0,1$ and ending in of the points $2,3$ are represented by 0. That is, if and only if $b=c=d=0$.\\
Here the condition for slope semistability reads as
$$\frac{\theta_0+\theta_1}{2}\leq \frac{\theta_0+\theta_1+\theta_2+\theta_3}{4},$$
which is equivalent to $0\leq -\theta_0-\theta_1+\theta_2+\theta_3$. The associated half space in $\chi(PG_d)_\R$ is thus given as $H(-1,-1,1,1)$. For a character in $H(-1,-1,1,1)$ this dimension vector does not impose any condition on the semistable locus. However, if the character is not contained in the half space, a necessary condition for a representation to be semistable is that this dimension vector does not occur, which is $b\neq 0$ or $c\neq 0$ or $d\neq 0$.\\

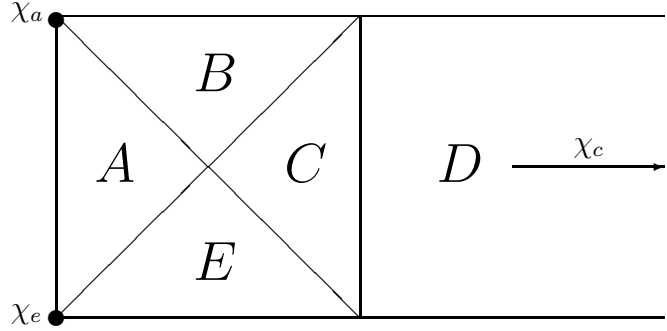
\begin{figure}[h]\label{fig:chambers}
\setlength{\unitlength}{2cm}
\begin{picture}(4,2.1)
\put(0,0){\line(1,1){1}}
\put(0,0){\line(1,0){4}}
\put(0,0){\line(0,1){2}}
\put(0,2){\line(1,-1){1}}
\put(0,2){\line(1,0){4}}
\put(2,2){\line(0,-1){2}}
\put(1,1){\line(1,1){1}}
\put(1,1){\line(1,-1){1}}
\put(0.25,0.9){\huge $A$}
\put(1.5,0.9){\huge $C$}
\put(2.5,0.9){\huge $D$}
\put(0.9,1.5){\huge $B$}
\put(0.9,0.25){\huge $E$}
\put(0,1.975){\circle*{0.1}}
\put(0,0){\circle*{0.1}}
\put(3,1){\vector(1,0){1}}
\put(-0.3,2){\large $\chi_a$}
\put(-0.3,0){\large $\chi_e$}
\put(3.4,1.1){\large $\chi_c$}
\end{picture}
\caption{These are the restrictions of the GIT chambers to the slice $S$, where the $x$- and $y$-axis correspond to coordinates on the slice and the $y$-axis is scaled with factor $\frac{1}{2}$. The origin is the point in which the chambers $A,B,C$ and $E$ meet. Additionally the directions associated to possible unstable codimension one components are given.}
\end{figure}

To compute which chambers are stable, we need to be able to translate directions in $\chi(PG_d)_\R$ to directions on the slice $S$. Hence suppose that we have a direction given as
$$c=c_p p+c_1 u_1+c_2 u_2\in\R^3,$$
where $c_p\geq 0$, and an arbitrary vector
$$x=p+su_1+tu_2\in S.$$
Then the translation of $x$ into the direction $c$ is given as
$$x+nc=(1+nc_p)p+(s+nc_1)u_1+(t+nc_2)u_2.$$
Obviously the ray spanned by $x+nd$ through the origin meets the slice $S$ in the point
$$p+\frac{s+nc_1}{1+nc_p}u_1+\frac{t+nc_2}{1+nc_p}u_2.$$
If $c_p\neq 0$, this converges to the point $(c_1,c_2)\in S$ in the chosen coordinates on $S$. But if $c_p=0$, the translation into the direction of $c$ corresponds to the translation into the direction of $(c_1,c_2)$ in the coordinates on $S$.\\

Now we need the characters associated to codimension one components of the unstable loci.\\
For the first coordinate function we have $g*a=\frac{g_0}{g_1}a$ (note that this is the contragredient action), which corresponds to the vector
$$\chi_a=(1,-1,0,0)\in \chi(G_d)_\R.$$
An elementary computation yields that this corresponds to a direction of the first type for $\chi_a=(-1,2)$ in coordinates on $S$.\\

With similar computation steps we obtain the directions $\chi_e=(-1,-2)$ of the first type, and $\chi_c=(1,0)$ of the second type.\\
This yields the following structure of stable chambers.\\

\begin{center}
\begin{tabular}{|c|c|c|}
\hline 
chamber & $\mathcal{Z}$ & stable chambers \\ 
\hline 
\hline
A & $N(a),N(e)$ & A \\ 
\hline 
B & $N(a)$ & A,B \\ 
\hline 
C & $\emptyset$ & all chambers \\ 
\hline 
D & $N(c)$ & D \\ 
\hline 
E & $N(e)$ & A,E \\ 
\hline 
\end{tabular}
\end{center}~\\

Finally, we compute some quotients associated to the chambers, showing that the Mori chamber structure of the quotients coincides with the stable GIT chamber structure.\\

For the chamber B, we claim that the associated quotient is the Hirzebruch surface $\F_1$. Indeed consider the maps
\begin{align*}
\varphi:R_d(Q)\to R_{d'}(Q'),~ & (a,b,c,d,e)\mapsto (e,b,ca,da), \\
\psi: R_{d'}(Q')\to R_d(Q),~ & (\alpha,\beta_1,\beta_2,\gamma)\mapsto (1,\beta_1,\beta_2,\gamma,\alpha),
\end{align*}
where the quiver $Q'$ and the dimension vector $d'$ are given as in Example \ref{ex:hirzebruch}. It is straightforward to check that $\varphi$ and $\psi$ respect the semistable loci of the chambers B and I respectively, and induce an isomorphism of the quotients.\\

Similarly the maps
\begin{align*}
\varphi(a,b,c,d,e)&=(eb,eca,da),\\
\psi(\alpha_1,\alpha_2,\alpha_3)&=(1,\alpha_1,\alpha_2,\alpha_3,1),
\end{align*}
where $Q'$ is given as
$$\xymatrix{A \ar[rr] \ar@/_/[rr] \ar@/^/[rr] && B,}$$
show that the quotient with respect to the chamber A is isomorphic to $\P^2$.\\

For these chambers the predictions of Theorem \ref{T:quotientmoridream} are thus confirmed. As a side remark we note that similar computations confirm the predictions for the chambers D and E, where the quotients are both nontrivial $\P^1$-bundles over $\P^1$.
\end{example}

We conclude with an example of a quotient of a Mori dream space.

\begin{example}
The group $G=SL_2$ acts on the Mori dream space $X=\left(\P^1\right)^6$ via the diagonalized canonical action on $\P^1$. Some quotients for this action are computed in \cite{polito}.\\
Following the notation given there, we consider quotients $X^s(m)/G$, where $m=(m_1,\ldots,m_6)$ indicates that we take the GIT quotient with respect to the bundle $\Ocal(m_1)\otimes\ldots\otimes \Ocal(m_6)$. Note that since the character group of $SL_2$ is trivial this bundle admits a unique $G$-linearization.\\
Corollary 5 in \cite{polito} shows that in the case when $|m|=m_1+\ldots+m_6$ is odd, the corresponding bundle lies in the interior of a GIT chamber, and we can use Theorem 3 of \cite{polito} to compute the unstable loci.\\

Three geometric quotients are explicitly given. They are well-known to be Mori dream spaces (even toric), and the formula of Theorem \ref{T:quotientofmds} correctly predicts the Picard number.
\begin{enumerate}
\item $X^s(2,3,3,1,1,1)/G\simeq Bl_{(1,1,1),(\infty,\infty,\infty)}\left(\P^1\right)^3$,\\
which has Picard number 5. Indeed, there is one unstable component of codimension one
$$Z=\left\{(x_1,\ldots,x_6)\in \left(\P^1\right)^6|x_2=x_3\right\}.$$
\item $X^s(2,2,2,1,1,1)/G\simeq Bl_{(0,0,0),(1,1,1),(\infty,\infty,\infty)}\left(\P^1\right)^3$,\\
which has Picard number 6. There are no unstable components of codimension one.
\item $X^s(1,1,2,1,1,1)/G\simeq Bl_{(0,0,0),(1,1,1),(\infty,\infty,\infty)}\left(\P^1\right)^3$,\\
so the quotient again has Picard number 6. In this case there are no unstable components of codimension one as well.
\end{enumerate}
\end{example}

\section{Appendix: Walls and chambers for Example \ref{ex:quiver}}\label{app:walls}

The 14 potential dimension vectors of nontrivial subrepresentations give the following 14 potential walls. We also state into which direction the associated half space extends.\\

\begin{center}
\begin{tabular}{|c|c|c|c|c|c|}
\hline 
 & $e$ & occurs iff & $a$ & $b$ & extends to  \\ 
 \hline
\hline 
I &  $(0,0,0,1)$ & (always) & $(0,2)$ & $(2,0)$ & neg. $y$-axis  \\ 
\hline 
II & $(0,0,1,0)$ & $e=0$ & $0$ & $(2,-4)$ & pos. $x$-axis  \\ 
\hline 
III & $(0,1,0,0)$ & $c=d=0$ & $0$ & $(2,4)$ & neg. $x$-axis  \\ 
\hline 
IV & $(1,0,0,0)$ & $a=b=0$ & $(0,-2)$ & $(-2,0)$ & neg. $y$-axis  \\ 
\hline 
V & $(1,1,0,0)$ & $b=c=d=0$ & $(-1,0)$ & $(0,2)$ & neg. $x$-axis  \\ 
\hline 
VI & $(1,0,1,0)$ & $e=a=0$ & $(1,0)$ & $(0,-2)$ & pos. $x$-axis  \\ 
\hline 
VII & $(1,0,0,1)$ & $a=b=0$ & $0$ & $(-2,0)$ & neg. $y$-axis  \\ 
\hline 
VIII & $(0,1,1,0)$ & $d=e=0$ & $0$ & $(2,0)$ & pos. $y$-axis \\ 
\hline 
IX & $(0,1,0,1)$ & $c=0$ & $(1,0)$ & $(0,2)$ & neg. $x$-axis \\ 
\hline 
X & $(0,0,1,1)$ & (always) & $(-1,0)$ & $(0,-2)$ & pos. $x$-axis \\ 
\hline 
XI & $(1,1,1,0)$ & $d=e=0$ & $(0,2)$ & $(2,0)$ & pos. $y$-axis \\ 
\hline 
XII & $(1,1,0,1)$ & $b=c=0$ & $0$ & $(-2,4)$ & neg. $x$-axis  \\ 
\hline 
XIII & $(1,0,1,1)$ & $a=0$ & $0$ & $(-2,-4)$ & pos. $x$-axis \\ 
\hline 
XIV & $(0,1,1,1)$ & (always) & $(0,-2)$ & $(2,0)$ & pos. $y$-axis  \\ 
\hline 
\end{tabular} 
\end{center}~\\

The stability conditions for the chambers, as given in Figure 1, read as follows.

\begin{align*}
(R_d(Q))^{A-sst}&=\left\{a\neq 0 \wedge e\neq 0 \wedge (b\neq0 \vee c\neq 0 \vee d \neq 0)\right\} \\
(R_d(Q))^{B-sst}&=\left\{a\neq 0 \wedge (b\neq 0 \vee c\neq 0) \wedge (d\neq 0 \vee e\neq 0)\right\} \\
(R_d(Q))^{C-sst}&=\left\{(a\neq 0\vee b\neq 0)\wedge(b\neq 0\vee c\neq 0)\wedge(c\neq 0\vee d\neq 0)\right.\\
&~~~~\left.\wedge(d\neq 0\vee e\neq 0)\wedge(e\neq 0\vee a\neq 0) \right\} \\
(R_d(Q))^{D-sst}&=\left\{c\neq 0 \wedge (a\neq 0 \vee b\neq 0) \wedge (d\neq 0 \vee e\neq 0) \right\} \\
(R_d(Q))^{E-sst}&=\left\{e\neq 0 \wedge (a\neq 0 \vee b\neq 0) \wedge (c\neq 0 \vee d\neq 0) \right\}.
\end{align*}

\end{document}